\documentclass[reqno,twoside,final]{amsart}
\usepackage{amsmath,amstext,amsthm,amsfonts,amssymb,amscd}

\newcommand{\etype}[1]{\renewcommand{\labelenumi}{(#1{enumi})}}
\def\eroman{\etype{\roman}}
\newcommand{\Z}{\mathbb{Z}}

\newcommand{\Image}{{\operatorname{Im}\ }}
\newcommand{\Char}{{\operatorname{Char}\ }}
\newcommand{\Cent}{{\operatorname{Cent}}}
\newcommand{\tr}{{\operatorname{tr}}}
\newcommand{\diag}{{\operatorname{Diag}}}

\newcommand{\GL}{{\operatorname{GL}}}

\newcommand{\N}{\mathbb{N}}

\theoremstyle{definition}
\newtheorem{definition}{Definition}

\newtheorem{rem}{Remark}
\newtheorem*{rem*}{Remark}
\newtheorem*{acknow*}{Acknowledgements}
\newtheorem*{examples*}{Examples}
\newtheorem{examples}{Example}
\theoremstyle{plain}

\newtheorem{lemma}{Lemma}
\newtheorem{prop}{Proposition}
\newtheorem{theorem}{Theorem}
\newtheorem*{theorem*}{Theorem}
\newtheorem{conjecture}{Conjecture}
\newtheorem{question}{Question}
\newenvironment{proof-sketch}{\noindent{\bf Sketch of Proof}\hspace*{1em}}{\qed\bigskip}
\newenvironment{proof-idea}{\noindent{\bf Proof Idea}\hspace*{1em}}{\qed\bigskip}
\newenvironment{proof-of-lemma}[1]{\noindent{\bf Proof of Lemma #1}\hspace*{1em}}{\qed\bigskip}
\newenvironment{proof-of-prop}[1]{\noindent{\bf Proof of Proposition #1}\hspace*{1em}}{\qed\bigskip}
\newenvironment{proof-of-thm}[1]{\noindent{\bf Proof of Theorem #1}\hspace*{1em}}{\qed\bigskip}
\newenvironment{proof-attempt}{\noindent{\bf Proof Attempt}\hspace*{1em}}{\qed\bigskip}

\begin{document}

\title[Images of non-commutative polynomials ]{The images of non-commutative polynomials evaluated on $2\times 2$ matrices.}
\author{Alexey Kanel-Belov, Sergey Malev, Louis Rowen}

\address{Department of mathematics, Bar Ilan University,
Ramat Gan, Israel} \email{beloval@math.biu.ac.il}
\email {malevs@math.biu.ac.il}
\email{rowen@math.biu.ac.il}
\thanks{This work was financially supported by the Israel Science
Foundation (grant no. 1178/06).
\newline The second named author was supported by an Israeli Ministry of Immigrant Absorbtion scholarship.
\newline We are grateful to
 Kulyamin, Latyshev, Mihalev,  E.~Plotkin, and L.~Small for
 useful comments. Latyshev and Mihalev indicated that the
 problem was originally posed by Kaplansky.}

\begin{abstract}
Let $p$ be a multilinear polynomial in several non-commuting
variables with coefficients in a quadratically closed field $K$ of
any characteristic. It has been conjectured that for any $n$, the
image of $p$ evaluated on the set $M_n(K)$ of $n$ by $n$ matrices
is either zero, or the set of scalar matrices, or the set
$sl_n(K)$ of matrices of trace 0, or all of $M_n(K)$. We prove the
conjecture for $n=2$, and show that although the analogous
assertion fails for completely homogeneous polynomials, one can
salvage the conjecture in this case by including the set of all
non-nilpotent matrices of trace zero and also permitting dense
subsets of $M_n(K)$.
\end{abstract}


\maketitle

\section{Introduction}
Images of polynomials evaluated on algebras  play  an important
role in non-commutative algebra. In particular, various important
problems related to the theory of polynomial identities have been
settled after the construction of central polynomials by Formanek
\cite{F1} and Razmyslov \cite{Ra1}.

 The
parallel topic in group theory (the images of words in
 groups) also has been studied extensively, particularly in recent
 years.
Investigation of the image sets of words in pro-$p$-groups is
related to the investigation of Lie polynomials and  helped
Zelmanov  \cite{Ze} to prove that the free pro-$p$-group cannot be
embedded in the algebra of $n\times n$ matrices when $p\gg n.$
(For $p>2$, the impossibility of  embedding the free pro-$p$-group
into the algebra of $2\times 2$ matrices had been proved by Zubkov
\cite{Zu}.) The general problem of nonlinearity of the free
pro-$p$-group is related on the one hand with images of Lie
polynomials and words in groups, and on the other hand with
problems of Specht type, which is of significant current interest.

Borel~\cite{B} (also cf.~\cite{La}) proved that for any simple
(semisimple) algebraic group~$G$ and any word $w$ of the free
group on $r$ variables, the word map $ w: G^r\to G$ is dominant.
Larsen and Shalev \cite{LaS} showed that any element of a simple
group can be written as a product of length two in the word map,
and Shalev \cite{S} proved Ore's conjecture, that the image of the
commutator word in a simple group is all of the group.

In this note we consider the question, reputedly raised by
Kaplansky, of the possible image set $\Image p$ of a polynomial
$p$ on matrices. When $p = x_1x_2-x_2x_1$, this is a theorem of
Albert and Muckenhoupt \cite{AM}. For an arbitrary polynomial, the
question was settled for the case when $K$ is a finite field
 by Chuang \cite{Ch}, who proved that a subset $S \subseteq
M_n(K)$ containing $0$ is the image of a polynomial with constant
term zero, if and only if $S$ is invariant under conjugation.
Later Chuang's result was generalized by Kulyamin \cite{Ku1},
\cite{Ku2} for graded algebras.

  Chuang \cite{Ch}  also observed that for an infinite field $K$,
if $\Image p$ consists only of nilpotent matrices, then $p$ is a
polynomial identity (PI). This can be seen via Amitsur's Theorem
\cite[Theorem 3.26, p.~176]{Row} that says that the relatively
free algebra of generic matrices  is a domain. Indeed, $p^n$ must
be a PI for $M_n(K)$, implying $p$ is a PI.

Lee and Zhou proved  \cite[Theorem 2.4]{LeZh} that when $K$ is an
infinite division ring, for any non-identity $p$ with coefficients
in the center of $K$, $\Image p$ contains an invertible matrix.

Over an infinite field, it is not difficult to ascertain the
linear span of the values of any polynomial. Indeed, standard
multilinearization techniques enable one to reduce to the case
where the polynomial $p$ is multilinear, in which case the linear
span
 of its values comprise a Lie ideal since, as is well-known,
 $$[a, p(a_1, \dots, a_n)]
 = p([a,a_1],a_2 \dots, a_n)+ p(a_1,[a,a_2] \dots, a_n)+ \cdots +p(a_1, \dots,[a,
 a_n]),$$ and Herstein \cite{Her} characterized Lie ideals of a
 simple ring $R$ as either being contained in the center or
 containing the commutator Lie ideal $[R,R].$ Another proof is
 given in \cite{BK}; also see Lemma~\ref{linear} below. It
is considerably more difficult to determine the actual image set
 $\Image p$, rather than its linear span.

 Thus, in \cite{Dn}, Lvov
formulated  Kaplansky's question as follows:

\begin{question}(I. Lvov)\label{lvov}
Let $p$ be a multilinear polynomial over a field $K$. Is the set
of values of $p$ on the matrix algebra $M_n(K)$ a vector space?
\end{question}

In view of the above discussion,   Question \ref{lvov} is
equivalent to the following:

\begin{conjecture}\label{Polynomial image}
If $p$ is a multilinear polynomial evaluated on the matrix ring
$M_n(K)$, then $\Image p$ is either $\{0\}$, $K$,  $sl_n(K)$, or
$M_n(K)$. Here $K$ is the set of scalar matrices and $sl_n(K)$ is
the set of matrices of trace zero.
\end{conjecture}
\begin{examples}\label{ex_multilin}
$\Image p$ can indeed   equal   $\{0\}$,  $K$,  $sl_n(K)$, or
$M_n(K)$. For example, if our polynomial is in one variable and
$p(x)=x$ then $\Image p=M_n(K)$. The image of the polynomial
$[x_1,x_2]$ is $sl_n(K)$. If the polynomial $p$ is central, then
its image is $K$ and examples of such polynomials can be found in
 \cite{Ra1} and in \cite{F1}. Finally if the
polynomial $p$  is a PI, then its image is $\{0\}$, and $s_{2n}$
is an example of such polynomial.
\end{examples}

As noted above, the conjecture fails for non-multilinear
polynomials when $K$ is a finite field.  The situation is
considerably subtler for images of non-multilinear, completely
homogeneous polynomials than for multilinear polynomials. Over any
field $K$, applying the structure theory of division rings to
Amitsur's theorem, it is not difficult to get an example of a
completely homogeneous polynomial $f$, noncentral on $M_3(K)$,
whose values all have third powers central; clearly its image does
not comprise a subspace of $M_3(K)$. Furthermore, in the
(non-multilinear) completely homogeneous case, the set of values
could  be dense without including all matrices. (Analogously,
although the finite basis problem for multilinear identities is
not yet settled in nonzero characteristic,  there are
counterexamples    for completely homogeneous polynomials,
cf.~\cite{Belov01}.)

Our main results in this note are for $n=2$, for which we settle
Conjecture \ref{Polynomial image}, proving the following results
(see \S\ref{def1} for terminology). We call a field $K$ {\it
quadratically closed} if every nonconstant polynomial of degree
$\le  2\deg p$ in $K[x]$ has a root in $K$.

\begin{theorem}\label{imhom}
Let $p(x_1,\dots,x_m)$ be a semi-homogeneous polynomial (defined
below) evaluated on the algebra $M_2(K)$ of $2\times 2$ matrices
over a quadratically closed field.
 Then $\Image p$ is either
 $\{0\}$, $K$, the set of all non-nilpotent matrices having
trace zero,
 $sl_2(K)$,  or  a
dense subset of $M_2(K)$ (with respect to Zariski topology).
\end{theorem}

(We also give examples to show how $p$ can have these images.)

\begin{theorem}
\label{main} If $p$ is a multilinear polynomial evaluated on the
matrix ring $M_2(K)$ (where $K$ is a quadratically closed field),
then $\Image p$ is either $\{0\}$,   $K$,  $sl_2$, or~$M_2(K)$.
\end{theorem}

Whereas for $2 \times 2$ matrices one has a full positive answer
for multilinear polynomials, the situation is ambiguous for
homogeneous polynomials, since, as we shall see, certain invariant
sets cannot occur as their images. For the general non-homogeneous
case, the image of a polynomial need not be dense, even if it is
non-central and takes on values of nonzero trace, as we see in
Example~\ref{nondense}. In this paper, we start with the
homogeneous case (which includes the completely homogeneous case,
then discuss the nonhomogeneous case, and finally give the
complete picture for the multilinear case.

The proofs of our theorems use some algebraic-geometric tools in
conjunction with ideas from  graph theory. The final part of the
proof of Theorem~\ref{main} uses the First Fundamental Theorem of
Invariant Theory (that in the case $\Char K=0$, invariant
functions evaluated on matrices are polynomials involving traces),
proved by Helling~\cite{Hel}, Procesi~\cite{P}, and Razmyslov
\cite{Ra3}. The formulation  in positive characteristic, due to
  Donkin ~\cite{D}, is somewhat more intricate.
$\GL_n(K)$ acts on $m-$tuples of $n\times n$-matrices by
simultaneous conjugation.
\begin{theorem*}[Donkin~\cite{D}]
For any $m, n\in \N$, the algebra of polynomial invariants
$K[M_n(K)^m]^{\GL_n(K)}$ under $\GL_n(K)$  is generated by the
trace functions
\begin{equation}\label{Donk} T_{i,j}(x_1,x_2,\dots,x_m) =
\operatorname{Trace}(x_{i_1}x_{i_2}\cdots
x_{i_r},\bigwedge\nolimits^jK^n),\end{equation} where
$i=(i_1,\dots,i_r),$ all $i_l\leq m,$ $r\in\N, j>0,$ and
$x_{i_1}x_{i_2}\cdots x_{i_r}$ acts as a linear transformation on
the exterior algebra $\bigwedge^jK^n$.

\end{theorem*}
\begin{rem*}
For $n=2$ we have a polynomial function in expressions of the form
$\operatorname{Trace}(A,\bigwedge^2K^2)$ and $\tr A$ where $A$ is
monomial. Note that $\operatorname{Trace}(A,\bigwedge^2K^2)=\det
A$.
\end{rem*}
(The Second Fundamental Theorem,   dealing with relations between
invariants, was proved by Procesi  \cite{P} and Razmyslov
\cite{Ra3} in the case $\Char K=0$ and by Zubkov \cite{Zu} in the
case $\Char K>0$.)

Other works on polynomial maps evaluated on matrix algebras
include \cite{W}, \cite{GK}, who investigated maps that preserve
zeros of multilinear polynomials.

\section{Definitions and basic preliminaries}\label{def1} By $K\langle x_1,\dots,x_m\rangle$ we
denote the free $K$-algebra generated by noncommuting variables
$x_1,\dots,x_m$, and refer to the elements of $K\langle
x_1,\dots,x_m\rangle$ as {\it polynomials}. Consider any algebra
$R$ over a field $K$. A polynomial $p\in K\langle
x_1,\dots,x_m\rangle$ is called a {\it polynomial identity} (PI)
of the algebra $R$ if $p(a_1,\dots,a_m)=0$ for all
$a_1,\dots,a_m\in R$;  $p\in K\langle x_1,\dots,x_m\rangle$ is a
{\it central polynomial} of $R$, if for any $a_1,\dots,a_m\in R$
one has $\mbox{$p(a_1,\dots,a_m)\in \Cent(R)$}$ but $p$ is not a
PI of $R$.
A polynomial $p$ (written as a sum of monomials) is called {\it
semi-homogeneous of weighted degree $d$} with (integer) {\it
weights} $(w_1,\dots,w_m)$  if for each monomial $h$ of $p$,
taking $d_j$ to be the degree of $x_{j}$ in $p$, we have
$$d_1w_1+\dots+d_nw_n=d.$$ A
semi-homogeneous polynomial with weights $(1,1,\dots, 1)$ is
called $\it{homogeneous}$ of degree $d$.

A polynomial $p$ is {\it completely homogeneous} of multidegree
$(d_1,\dots,d_m)$ if each variable $x_i$ appears the same number
of times $d_i$ in all monomials.
A polynomial $p\in K\langle x_1,\dots,x_m\rangle$ is called {\it
multilinear} of degree $m$ if it is linear (i.e. homogeneous of
multidegree $(1,1,\dots,1)$). Thus, a polynomial is multilinear if
it is a polynomial of the form
$$p(x_1,\dots,x_m)=\sum_{\sigma\in S_m}c_\sigma
x_{\sigma(1)}\cdots x_{\sigma(m)},$$ where $S_m$ is the
symmetric group in $m$ letters and the coefficients $c_\sigma$ are
constants in $K$.

 We need a slight
modification of Amitsur's theorem, which is well known:
\begin{prop}\label{Am1} The algebra of generic matrices with traces
is a domain which can be embedded in the division algebra UD of
central fractions of Amitsur's algebra of generic matrices.
Likewise, all of the functions in Donkin's theorem can be embedded
in UD.
\end{prop}
\begin{proof} Any trace function can be expressed as the ratio of two
central polynomials, in view of \cite[Theorem 1.4.12]{Row}; also
see \cite[Theorem~J, p.~27]{BR} which says for any characteristic
coefficient $\alpha_k $ of the characteristic polynomial
$\lambda^t + \sum_{k=1}^t (-1)^k \alpha _k \lambda ^{t-k}$ that
\begin{equation}\label{trace2pol0}
\alpha_k f(a_1, \dots, a_t, r_1, \dots, r_m) = \sum _{k=1}^t
f(T^{k_1}a_1, \dots, T^{k_t} a_t,  r_1, \dots, r_m) ,
\end{equation}
summed over all vectors $(k_1, \dots, k_t)$ where each $k_i \in \{
0, 1 \}$ and $\sum k_i = t,$ where $f$ is any $t-$alternating
polynomial (and $t = n^2$). In particular,
\begin{equation}\label{trace2pol}
\tr(T)f(a_1, \dots, a_t, r_1, \dots, r_m) = \sum _{k=1}^t f(a_1,
\dots, a_{k-1}, Ta_k, a_{k+1} , \dots, a_t,  r_1, \dots, r_m) ,
\end{equation}
so any trace of a polynomial belongs to UD. In general, the
function \eqref{Donk} of Donkin's theorem can be written as a
characteristic equation, so we can apply Equation~
\eqref{trace2pol0}.
\end{proof}

Here is one of the main tools for our investigation.

\begin{definition}
A {\it cone} of $M_n(K)$ is a subset closed under multiplication
by nonzero constants. An  {\it invariant cone} is a cone invariant
under conjugation. An invariant cone  is  {\it irreducible} if it
does not contain any nonempty  invariant cone.
\end{definition}

\begin{examples}\label{coneex} Examples of  invariant cones of $M_n(K)$ include:
\begin{enumerate}\eroman  \item The set of diagonalizable matrices.
 \item The set of non-diagonalizable matrices.
  \item The set $K$ of scalar matrices.
  \item The set of nilpotent matrices.
    \item The set $sl_n$ of  matrices having trace zero.
\end{enumerate}
\end{examples}

\section{Images of Polynomials}
For any polynomial $p\in K\langle x_1,\dots,x_m\rangle$, the
$image$ of $p$ (in $R$) is defined as
$$\Image p=\{A\in R:
\ \text{there exist}\ a_1,\dots,a_m\in R\ \text{such that}\
p(a_1,\dots,a_m)=A \}.$$

\begin{rem}\label{cong}
$\Image p$ is invariant under conjugation, since $$\alpha
p(x_1,\dots,x_m)\alpha^{-1}=p(\alpha x_1\alpha^{-1},\alpha
x_2\alpha^{-1},\dots,\alpha x_m\alpha^{-1})\in\Image p,$$ for any
nonsingular $\alpha\in M_n(K)$.
\end{rem}

\begin{lemma}\label{nilp} If $\Char K$ does not divide  $n$, then any non-identity $p(x_1, \dots, x_m)$ of $M_n(K)$ must either be a central polynomial
or take on a value which is a matrix whose eigenvalues are not all
the same.\end{lemma}
\begin{proof} Otherwise $p(x_1, \dots, x_m)- \frac 1n \tr(p(x_1, \dots,
x_m))$
is a nilpotent element in the algebra of generic matrices with
traces, so by Proposition~\ref{Am1} is 0, implying $p$ is central.
\end{proof}

 Let us continue with the following easy but crucial
lemma.

\begin{lemma}\label{cone_conj} Suppose the field $K$ is closed
under $d$-roots. If the image of a semi-homogeneous polynomial $p$
of weighted degree $d$  intersects an irreducible invariant cone
$C$ nontrivially, then $C \subseteq \Image p$.
\end{lemma}
\begin{proof}

If $A\in\Image p$ then $A=p(x_1,\dots,x_m)$ for some $x_i\in
M_n(K)$. Thus for any $c\in K$,
$cA=p(c^{w_1/d}x_1,c^{w_2/d}x_2,\dots,c^{i_m/d}x_m) \in\Image p$,
where $(w_1,\dots,w_m)$ are the weights. This shows that $\Image
p$ is a cone.
\end{proof}
\begin{rem} When the polynomial $p$ is multilinear, we take the
weights $w_1 = 1$ and $w_i = 0$ for all $i>1,$ and thus  do not
need any assumption on $K$  to show that the image of any
multilinear polynomial is an invariant cone.
\end{rem}
\begin{lemma}\label{scalar}
If  $\Image p$ consists only of diagonal matrices, then the image
$\Image p$ is either~$\{0\}$ or the set $K$ of scalar matrices.
\end{lemma}
\begin{proof}
Suppose that some nonscalar diagonal matrix
$A=\diag\{\lambda_1,\dots,\lambda_n\}$ is in the image. Therefore
$\lambda_i\neq\lambda_j$ for some $i$ and $j$. The matrix
$A'=A+e_{ij}$ (here $e_{ij}$ is the matrix unit) is conjugate to
$A$ so by Remark~\ref{cong} also belongs to $\Image p$. However
$A'$ is not diagonal, a contradiction.
\end{proof}
\begin{lemma}\label{graph}
Assume that the $x_i$ are matrix units. Then $p(x_1,\dots,x_m)$ is
either $0$, or~$c\cdot e_{ij}$ for some $i\neq j$, or a diagonal
matrix.
\end{lemma}
\begin{proof}
Suppose that the $x_i$ are matrix units $e_{k_i,l_i}$. Then the
product $x_1\cdots x_m$ is nonzero if and only if $l_i=k_{i+1}$
for each $i$, and in this case this product is equal to
$e_{k_1,l_m}$. If $x_i$ are such that there is at least one
$\sigma\in S_n$ such that $x_{\sigma(1)}\cdots x_{\sigma(m)}$ is
nonzero then we can consider a graph on $n$ vertices whereby we
connect vertex $i$ with vertex $j$ by an oriented edge if there is
a matrix $e_{ij}$ in our set $\{x_1,x_2,\dots,x_m\}$. It can
happen that we will have more than one edge that connects $i$ to
$j$ and it is also possible that we will have edges connecting a
vertex to itself.  The evaluation $p(x_1,\dots,x_m)\neq 0$ only if
there exists an Eulerian cycle or an Eulerian path in the graph.
This condition is necessary but is not sufficient. From graph
theory we know that there exists an Eulerian path only if the
degrees of all vertices but two are even, and the degrees of these
two vertices are odd. Also we know that there exists an Eulerian
cycle only if the degrees of all vertices are even. Thus when
$p(x_1,\dots,x_m)\neq 0$,  there exists either an Eulerian path or
cycle in the graph. In the first case we have exactly two vertices
of odd degree such that one of them ($i$) has more output edges
and another ($j$) has more input edges. Thus the only nonzero
terms in the sum of our polynomial can be of the type $ce_{ij}$
and therefore the result will also be of this type. In the second
case all degrees are even. Thus there are only cycles and the
result must be a diagonal matrix.
\end{proof}

As mentioned earlier, the following result follows easily
from~\cite{Her}, with another proof given in \cite{BK}, but a
 self-contained proof is included here for completeness.
\begin{lemma}\label{linear}
If the image of $p$ is not $\{0\}$ or the set of scalar matrices
then for any $i\neq j$ the matrix unit $e_{ij}$ belongs to $\Image
p$. The linear span of   $L = \Image p$ must be either $\{0\}$,
$K$,  $sl_n$, or $M_n(K)$.
\end{lemma}
\begin{proof}
Assume that the image is neither $\{0\}$ nor the set of scalar
matrices. Then by Lemma \ref{scalar} the image contains a
nondiagonal matrix  $p(x_1,\dots,x_m)=A$. Any $x_i$ is a linear
combination of matrix units. After opening brackets on the left
hand side we will have a linear combination of evaluations of $p$
on matrix units, and on the right hand side a nondiagonal matrix.
From Lemma \ref{graph} it follows that any evaluation of $p$ on
matrix units is either diagonal or a matrix unit multiplied by
some coefficient. Thus there is a matrix $e_{ij}$ for $i\neq j$ in
$\Image p$. Since any nondiagonal $e_{kl}$ is conjugate to
$e_{ij}$,
  all nondiagonal matrix units belong to the image. Thus all
matrices with zeros on the diagonal belong to the linear span of
the image. Taking matrices conjugate to these,
we obtain $sl_n\subseteq L$. Thus $L$ must be either $sl_n$ or
$M_n$.
\end{proof}
\subsection{The case $M_2(K)$}\label{n=2}
Now we consider the case $n=2$. We start by introducing the cones
of main interest to us, drawing from Example~\ref{coneex}.

\begin{examples}\label{coneex1} $ $

\begin{enumerate}\eroman
\item The set of nonzero nilpotent matrices comprise an
  irreducible invariant cone, since these all have the same
minimal and characteristic polynomial~$x^2$.
 \item The set of nonzero
 scalar
  matrices is an irreducible invariant cone.
 \item Let $\tilde K$  denote the set of non-nilpotent,
non-diagonalizable matrices in~$M_2(K).$ Note that $A\in \tilde K$
precisely when  $A$ is non-scalar, but with equal nonzero
eigenvalues, which is the case if and only if $A$ is the sum of a
nonzero scalar matrix with a nonzero nilpotent matrix. These are
all conjugate when the scalar part is the identity, i.e., for
matrices of the form $$\left(\begin{array}{cc}1 & a \\0 &
1\end{array}\right),\qquad a \ne 0$$ since these all have the same
minimal and characteristic polynomials, namely $x^2 - 2x+ 1.$ It
follows that $\tilde K$ is an irreducible invariant cone.
   \item Let $\hat K$  denote the set of non-nilpotent
  matrices in~$M_2(K)$ that have trace zero.

  When $\Char K \ne 2$, $\hat K$  is an irreducible invariant
cone,
     since any such matrix has distinct eigenvalues and thus is conjugate to
$\left(\begin{array}{cc}\lambda & 0\\0 &
-\lambda\end{array}\right)$.

 When $\Char K = 2$, $\hat K$ is an irreducible invariant cone,
     since any such matrix  is conjugate to
$\left(\begin{array}{cc}\lambda & 1\\0 & \lambda
\end{array}\right)$.
 \item  $sl_2(K)\setminus \{ 0 \}$ is the union of the two irreducible invariant cones of (i) and
 (iv). (The cases $\Char K \ne 2$  and
 $\Char K = 2$ are treated separately.)

 \item Let $C$ denote the set of nonzero  matrices which are the sum of a  scalar and a nilpotent
 matrix. Then $C$ is the union of the following
 three irreducible invariant cones: The nonzero scalar matrices,
 the nilpotent matrices, and the nonzero scalar multiples of
 non-identity
 unipotent matrices. (All  non-identity unipotent matrices are conjugate.)

\end{enumerate}
\end{examples}

 From now
on, we assume that $K$ is a quadratically closed field. In
particular, all of the eigenvalues of a matrix $A\in M_2(K)$ lie
in $K$. One of our main ideas is to consider some invariant of the
matrices in $\Image(p)$, and study the corresponding invariant
cones. Here is the first such invariant that we consider.

\begin{rem}\label{Phi}
Any non-nilpotent $2\times 2$ matrix $A$ over a quadratically
closed field has two eigenvalues $\lambda_1$ and $\lambda_2$ such
that at least one of them is nonzero. Therefore one can define the
ratio of eigenvalues, which is well-defined up to taking
reciprocals: $\frac{\lambda_1}{\lambda_2}$ and
$\frac{\lambda_2}{\lambda_1}.$ Thus, we will say that two
non-nilpotent matrices have {\it different ratios} of eigenvalues
if and only if their ratios of eigenvalues are not equal nor
reciprocal.

We do have a well-defined mapping $\Pi:\ M_2(K)\rightarrow K$
given by $A\mapsto
\frac{\lambda_1}{\lambda_2}+\frac{\lambda_2}{\lambda_1}$.  This
mapping is algebraic because
$$\frac{\lambda_1}{\lambda_2}+\frac{\lambda_2}{\lambda_1}=-2+\frac{(\tr A)^2}{\det A}.$$
\end{rem}

\begin{rem}\label{Phi2} The set of non-scalar
 diagonalizable matrices with a fixed nonzero ratio $r$ of eigenvalues (up to taking reciprocals) is an irreducible invariant cone. Indeed, this
is true since any such diagonalizable matrix is conjugate to
$$\lambda\left(\begin{array}{cc} 1 & 0\\0 & r\end{array}\right).$$
\end{rem}

\subsection{Images of semi-homogeneous polynomials}

We are ready to   prove Theorem~\ref{imhom}.

\begin{lemma}\label{linear1}
Suppose $K$ is closed under $d$-roots, as well as being
quadratically closed. If the image $\Image p$ of a
semi-homogeneous polynomial $p$ of weighted degree $d$ contains an
element of $\tilde K$, then $\Image p$ contains all of $\tilde K$.
\end{lemma}
\begin{proof} This is clear from
Lemma~\ref{cone_conj}(iii) together with Example~\ref{coneex1},
since $\tilde K$ is an irreducible invariant cone.
\end{proof}
 {\bf Proof of
Theorem~\ref{imhom}.} Assume that there are matrices
$p(x_1,\dots,x_m)$ and $p(y_1,\dots,y_m)$ with different ratios of
eigenvalues in the image of $p$. Consider the polynomial matrix
$f(t)=p(tx_1+(1-t)y_1,tx_2+(1-t)y_2,\dots,tx_m+(1-t)y_m)$, and
$\Pi\circ f$ where $\Pi$ is defined in Remark~\ref{Phi}. Write
this nonconstant rational function $\frac{\tr^2 f}{\det f}$ in
lowest terms as $\frac{A(t)}{B(t)}$, where $A(t), B(t)$ are
polynomials of degree $\leq 2\deg p$ in the numerator and
denominator.

An element $c\in K$ is in $\Image (\Pi\circ f)$ iff there exists
$t$ such that $A-cB=0$ (If for some
$t^*$ $A(t^*)-cB(t^*)=0$, then 
$t^*$ would be a common root of $A$ and $B$). Let $d_c =
\deg(A-cB)$. Then $d_c \le \max(\deg A,\deg B)\leq 2\deg p$, and
$d_c=\max(\deg A,\deg B)$ for almost all $c$. Hence, the
polynomial $A-cB$ is not
  constant and thus there is a root. Thus the image of $\frac{A(t)}{B(t)}$ is Zariski
dense, implying the image of $\frac{\tr^2 f}{\det f}$ is Zariski
dense.

Hence, we may assume that  $\Image p$ consists only of matrices
having a fixed ratio $r$ of eigenvalues. If $r \ne \pm 1$, the
eigenvalues
 $\lambda_1$ and $\lambda_2$  are linear functions of $\tr\
p(x_1\dots,x_m).$ Hence $\lambda_1$ and $\lambda_2$ are the
elements of the algebra of generic matrices with traces,  which is
a domain by Proposition~\ref{Am1}. But the two nonzero elements
$p-\lambda_1I$ and $p-\lambda_2I$ have product zero, a
contradiction.

We conclude that $r = \pm 1.$ First assume $r = 1$. If $\Char K
\ne 2$, then $p$ is a PI, by~Lemma~\ref{nilp}. If $\Char K=2$ then
the image is
either $sl_2(K)$ or $\hat K$, by Example~\ref{coneex1}(v).

Thus, we may assume $r = -1$ and $\Char K\neq 2$. Hence, $\Image
p$ consists only of matrices with $\lambda_1=-\lambda_2$. By
Lemma~\ref{nilp}, there is a non-nilpotent matrix in the image of
$p$. Hence, by Example~\ref{coneex1}(v), $\Image p$ is either
$\hat K$  or strictly contains it and is all of $sl_2(K)$. \hfill
$\square$

\bigskip
We illuminate this result with some examples to show that certain
cones may be excluded.
\begin{examples}\label{coneex2}$ $
\begin{enumerate}\eroman  \item
The polynomial $g(x_1, x_2) = [x_1, x_2]^2$  has the property that
$g(A,B) = 0$ whenever $A$  is scalar, but $g$ can take on a
nonzero value whenever $A$  is non-scalar.
 Thus, $g(x_1, x_2)x_1$ takes on all values except
scalars. This polynomial is completely homogeneous, but not
multilinear. (One can linearize in $x_2$ to make $g$ linear in
each variable except $x_1,$ and the same idea can be applied to
Formanek's construction \cite{F1} of a central polynomial for any
$n$.)

\item Let $S$ be any finite subset of $K$. There exists a
completely homogeneous polynomial $p$ such that $\Image p$ is the
set of all $2\times 2$ matrices except the matrices with ratio of
eigenvalues from $S$. The construction is as follows. Consider
$$f(x)=x\cdot\prod_{\delta\in
S}(\lambda_1-\lambda_2\delta)(\lambda_2-\lambda_1\delta),$$ where
$\lambda_{1,2}$ are eigenvalues of $x$. For each $\delta$ the
product $(\lambda_1-\lambda_2\delta)(\lambda_2-\lambda_1\delta)$
is
a polynomial of $\tr\ x$ and $\tr\ x^2$. Thus $f(x)$ is a
polynomial with traces, and, as noted above (by \cite[Theorem
1.4.12]{Row}), one can rewrite each trace in $f$ as a fraction of
multilinear central polynomials (see \eqref{trace2pol} in
Proposition~\ref{Am1}).
After that we multiply the expression by the product of all the
denominators, which we can take to have value 1. We obtain a
completely homogeneous polynomial $p$ which image is the cone
under $\Image f$ and thus equals $\Image f$. The image of $p$ is
the set of all non-nilpotent matrices with ratios of eigenvalues
not belonging to $S$.

\item The image of a completely homogeneous polynomial evaluated
on $2\times 2$ matrices can also be $\hat K$. Take
$f(x,y)=[x,y]^3$. This is the product of $[x,y]^2$ and $[x,y]$.
$[x,y]^2$ is a central polynomial, and therefore $\tr\ f=0.$
However, there are no nilpotent matrices in $\Image p$ because if
$[A,B]^3$ is nilpotent then $[A,B]$ (which is a scalar multiple of
$[A,B]^3$) is nilpotent and therefore $[A,B]^2=0$ and $[A,B]^3=0$.

\item Consider the polynomial $$\qquad p(x_1,x_2,y_1,y_2) =
[(x_1x_2)^2,(y_1y_2)^2]^2 +
[(x_1x_2)^2,(y_1y_2)^2][x_1y_1,x_2y_2]^2.$$ Then $p$ takes on all
scalar values (since it becomes central by specializing $x_1
\mapsto x_2$ and $y_1 \mapsto y_2$), but also takes on all
nilpotent values, since specializing $x_1  \mapsto I + e_{12}$,
$x_2 \mapsto e_{22}$, and $y_1 \mapsto
 e_{12}$, and $y_2 \mapsto
e_{21} $ sends $p$~to $$\qquad \qquad [(e_{12}+e_{22})^2,
e_{11}^2]^2 + [(e_{12}+e_{22})^2, e_{11}^2][ e_{12}, e_{21}] \\ =
0 - e_{12}( e_{11}- e_{22}) =  e_{12}.$$

We claim that $\Image p$ does not contain any matrix~$a= p(\bar
x_1,\bar x_2,\bar y_1,\bar y_2)$ in ~$ \tilde K$. Otherwise, the
matrix
 $[(\bar x_1\bar x_2)^2,(\bar y_1\bar y_2)^2][\bar x_1\bar y_1,\bar x_2\bar y_2]^2$ would be the difference
of a matrix having equal eigenvalues and a scalar matrix, but of
trace 0, and so would have both eigenvalues 0 and thus be
nilpotent. Thus $[(\bar x_1\bar x_2)^2,(\bar y_1\bar y_2)^2]$
would also be nilpotent, implying the scalar term $[(\bar x_1\bar
x_2)^2,(\bar y_1\bar y_2)^2]^2$ equals zero, implying $a$ is
nilpotent, a contradiction.

$\Image p$ also contains all matrices having two distinct
eigenvalues. We conclude that $\Image p = M_2(K) \setminus \tilde
K.$
\end{enumerate}
\end{examples}

\begin{rem} In Example 4(iv),
The intersection $S$ of $\Image p$ with the discriminant surface
is defined by the polynomial $\tr(p(x_1, \dots, x_m))^2- 4
\det(p(x_1, \dots, x_m))= (\lambda_1 -\lambda_2)^2$.   $S$ is the
union of two irreducible varieties (its scalar matrices and the
nonzero nilpotent matrices), and thus $S$ is a reducible variety.
Thus, we see that the discriminant surface of a polynomial $p$ of
the algebra of generic matrices can be reducible, even if it is
not divisible by any trace polynomial. Such an example could not
exist for $p$ multilinear, since then, by the same sort of
argument as given in the proof of Theorem~\ref{imhom}, the
discriminant surface would give a generic zero divisor in
Amitsur's universal division algebra UD of Proposition~\ref{Am1},
a contradiction. In fact, we will also see that the image of a
multilinear polynomial cannot be as in Example~4(iv).
\end{rem}


\subsection{Images of non-homogeneous polynomials}

Now we consider briefly the general case.  One can write any
polynomial
  $p(x_1,\dots,x_m)$ as $p=h_k+\dots+h_n,$ where the
$h_i$ are semi-homogeneous polynomial of weighted degree $i$.

\begin{prop} Notation as above,
assume that there are weights $(w_1,\dots,w_m)$ that   the
polynomial $h_n$ has image dense in $M_2(K)$. Then $\Image p$ is
dense in $M_2(K).$
\end{prop}

\begin{proof}  Consider
$$p(\lambda^w_1x_1,\dots,\lambda^w_mx_m)=\sum\limits_{i=k}^n
h_i\lambda^i.$$ One can write $\tilde
P=\lambda^{-n}p(\lambda^w_1x_1,\dots,\lambda^w_mx_m)$ as a
polynomial in $x_1,\dots,x_m$ and $\varepsilon=\frac{1}{\lambda}$.
The matrix polynomial is the set of four polynomials
$p_{1,1},p_{1,2},p_{2,1},p_{2,2},$ which we claim are independent.
If there is some polynomial $h$ in four variables such that
$h(p_{1,1},p_{1,2},p_{2,1},p_{2,2})=0$ then $h$ should vanish on
four polynomials of $\tilde P$ for each $\varepsilon$, in
particular for $\varepsilon=0,$ a contradiction.
\end{proof}

\begin{rem} The case remains open where
  $p(x_1,\dots,x_m)$ is a  polynomial for which there are no
weights $(w_1,\dots,w_m)$ such that one can write
$p=h_k+\dots+h_n,$ where $h_i$ is  semi-homogeneous of weighted
degree $i$ and $h_n$ has image dense in $M_2$.\end{rem}

\begin{examples}\label{nondense} For $\Char K\neq 2$ we give an example of such a polynomial
whose middle term has image dense in $M_2(K)$. Take the polynomial
$f(x,y)=[x,y]+[x,y]^2$. It is not hard to check that $\Image f$ is
the set of all matrices with eigenvalues $c^2+c$ and $c^2-c$.
Consider
$p(\alpha_1,\alpha_2,\beta_1,\beta_2)=f(\alpha_1+\beta_1^2,\alpha_2+\beta_2^2)$.
The polynomials $f$ and $p$ have the same images. Now let us open
the brackets. The  term of degree 4 is
$h_4=[\alpha_1,\alpha_2]^2+[\beta_1^2,\beta_2^2]$. The image of
$h_4$ is all of $M_2(K)$, because $[\alpha_1,\alpha_2]^2$ can be
any scalar matrix and $[\beta_1^2,\beta_2^2]$ can be any trace
zero matrix. However the image of $p$ is the set of all matrices
with eigenvalues $c^2+c$ and $c^2-c$.
\end{examples}

\subsection{Images of multilinear polynomials}
\begin{lemma}\label{2two}
If $A,B\in  \Image p$ have different   ratios of   eigenvalues,
then $\Image p$ contains matrices having arbitrary ratios of
eigenvalues $\frac{\lambda_1}{\lambda_2}\in K.$
\end{lemma}
\begin{proof}
If $A = p(x_1,\dots,x_m), \ B = p(y_1,\dots,y_m)\in  \Image p$
have different ratios of   eigenvalues, then we can lift the
$x_1,\dots,x_m, y_1,\dots,y_m$ to generic matrices, and then
$p(x_1,\dots,x_m)=\tilde A$ and $p(y_1,\dots,y_m)=\tilde B$ also
have different ratios of eigenvalues. Then take
$$f(T_1,T_2,\dots,T_m)=p(\tau_1 x_1+t_1y_1,\dots,\tau_m x_m+t_my_m),$$
where $T_i=(t_i,\tau_i)\in K^2.$ The polynomial $f$ is linear with
respect to all $T_i$.

In view of Remark~\ref{Phi}, it is enough to show that the ratio
$\frac{(\tr f)^2}{\det f}$ takes on all values. Fix
$T_1,\dots,T_{i-1},T_{i+1},\dots,T_m$ to be generic pairs where
$i$ is such that $\frac{(\tr f)^2}{\det f}$ is not constant with
respect to $T_i$. Such $i$ exist because otherwise all matrices in
the image (in particular, $A$ and $B$) have the same ratio of
eigenvalues.
 But
$\frac{(\tr f)^2}{\det f}$ is the ratio of quadratic polynomials,
and $K$ is quadratically closed.

 If there is a point $T_i$ such
that $\tr f=\det f=0$, then   $f$ evaluated at this  $T_i$ is
nilpotent. Since $\tr f$ is a linear function,  the equation  $\tr
f = 0$ has only one root, which is a rational function on the
other parameters. Thus $f$  evaluated at this  $T_i$ is 0, by
Amitsur's Theorem. We conclude that the ratio of eigenvalues does
not depend on $T_i$, contrary to our assumption on $i$. Hence, we
can solve $\frac{(\tr f)^2}{\det f} =c$ for any $c \in K$.
\end{proof}

\begin{lemma}\label{2general_one}
If there exist $\lambda_1\neq\pm\lambda_2$ with a collection of
matrices $(E_1,E_2,\dots,E_m)$ such that $p(E_1,E_2,\dots,E_m)$
has eigenvalues $\lambda_1$ and $\lambda_2$, then all
diagonalizable matrices lie in $\Image p$.
\end{lemma}
\begin{proof}
Applying Lemma~\ref{graph} to the hypothesis, there is a matrix
$$\left(\begin{array}{cc}\lambda_1 &
0 \\0 & \lambda_2\end{array}\right)\in\Image p,\
\lambda_1\neq\pm\lambda_2$$ which is an evaluation of $p$ on
matrix units $e_{ij}.$ Consider the following mapping $\chi$
acting on the indices of the matrix units:
$\chi(e_{ij})=e_{3-i,3-j}$. Now take the polynomial
$$f(T_1,T_2,\dots,T_m)=p(\tau_1 x_1+t_1\chi(x_1),\dots,\tau_m x_m+t_m\chi(x_m)),$$
where $T_i=(t_i,\tau_i)\in K^2,$ which is linear with respect to
each $T_i$. Let us open the brackets. We obtain $2^m$ terms and
for each of them the degrees of all vertices stay even. (The edge
$12$ becomes $21$ which does not change degrees, and the edge $11$
becomes $22$, which decreases the degree of the vertex $1$ by two
and increases the degree of the vertex $2$ by two.) Thus all terms
remain diagonal. Consider generic pairs $T_1,\dots,T_m\in K^2.$
For each $i$ consider the polynomial $\tilde
f_i(T_i^*)=f(T_1,\dots,T_{i-1},T_i+T_i^*,T_{i+1},\dots,T_m)$. For
at least one $i$ the ratio of eigenvalues of $\tilde f_i$ must be
different from $\pm 1$. (Otherwise the ratio of eigenvalues  of
$\tilde f_i$ equal $\pm 1$ all $i$, implying $\lambda_1= \pm
\lambda_2\}$, a contradiction.)

  Fix $i$ such that the ratio of eigenvalues of $\tilde f_i$ is not $\pm 1$.
  By linearity, $\Image(\tilde f_i)$ takes on values with all possible ratios of
  eigenvalues; hence,
the cone under $\Image(\tilde f_i)$ is the set of all diagonal
matrices. Therefore by Lemma \ref{cone_conj} all diagonalizable
matrices lie in the image of $p$.
\end{proof}

\begin{lemma}
\label{thm1} If $p$ is a multilinear polynomial evaluated on the
matrix ring $M_2(K)$, then $\Image p$ is either $\{0\}$, $K$,
$sl_2$,  $M_2(K)$, or $M_2(K)\setminus\tilde K.$
\end{lemma}
\begin{proof} In view of Lemma~\ref{cone_conj}, we are done unless $\Image p$ contains
a non-scalar matrix. By Lemma \ref{linear} the linear span of
$\Image p$ is  $sl_2$  or $M_2(K)$. We treat the characteristic 2
and characteristic $\ne 2$ cases separately.

CASE I: $\Char K=2$. Consider the set
$$\Theta= \{ p(E_1,\dots,E_m)  \quad \text{where the } E_j \text{ are matrix units}\}.$$
 If
the linear span of the image is not $sl_2$, then $\Theta$ contains
 at least one non-scalar diagonal matrix
$\diag\{\lambda_1,\lambda_2\}$, so $\lambda_1\neq -\lambda_2$
(since $+1 = -1$). Hence~by Lemma \ref{2general_one}, all
diagonalizable matrices belong to $\Image p$. Thus, $\Image p$
contains $M_2(K)\setminus\tilde K$.

If the linear span of the image of $p$ is $sl_2$, then by Lemma
~\ref{graph} the identity matrix (and thus all scalar matrices)
and $e_{12}$ (and thus all nilpotent matrices) belong to the
image. On the other hand, in characteristic 2, any matrix $sl_2$
is conjugate to a matrix of the form $\lambda _1 I + \lambda_2
e_{1,2},$ and we consider the invariant $\frac
{\lambda_2}{\lambda_1}.$ Take $x_1,\dots,x_m$ to be generic
matrices. If $p(x_1,\dots,x_m)$ were nilpotent then $\Image p$
would consist only of nilpotent matrices, which is impossible. By
Example~\ref{coneex1}(v), $p(x_1,\dots,x_m)$ is not scalar and not
nilpotent, and thus is a matrix from $\tilde K$. Hence, $\tilde K
\subset \Image p,$ by~ Lemma~\ref{linear1}. Thus, all trace zero
matrices belong to $\Image p.$

CASE II: $\Char K\neq 2$. Again assume that the image is not
$\{0\}$ or the set of scalar matrices. Then by Lemma \ref{linear}
we obtain that $e_{12}\in\Image p$. Thus all nilpotent matrices
lie in $\Image p$. If the image consists only of matrices of trace
zero, then by Lemma \ref{linear} there is at least one matrix in
the image with a nonzero   diagonal entry. By Lemma~\ref{graph}
there is a set of matrix units that maps to a nonzero diagonal
matrix which, by assumption, is of trace zero and thus is
$\left(\begin{array}{cc}c & 0
\\0 & -c\end{array}\right).$ By Lemma \ref{cone_conj} and Example~\ref{coneex1}, $\Image p$ contains
 all  trace zero $2\times 2$ matrices.

Assume that the image
contains a matrix with nonzero trace. Then by Lemma~\ref{linear}
the linear span of the image is $M_2(K)$, and together with Lemma
\ref{graph} we have at least two diagonal linearly independent
matrices in the image. Either these matrices have ratios of
eigenvalues $(\lambda_1:\lambda_2)$ and $(\lambda_2:\lambda_1)$
for $\lambda_1\neq\pm \lambda_2$ or these matrices have
non-equivalent ratios. In the first case we can use Lemma
\ref{2general_one} which says that all diagonalizable matrices lie
in the image. If at least one of these matrices have ratio not
equal to $\pm 1$, then in the second case we also use Lemma
\ref{2general_one} and obtain that all diagonalizable matrices lie
in the image. If these matrices are such that the ratios of their
eigenvalues are respectively $1$ and $-1$, then we use Lemma
\ref{2two} and obtain that all diagonalizable matrices with
distinct eigenvalues lie in the image. By assumption, in this
case, scalar matrices also belong to the image. Therefore we
obtain that for any ratio $(\lambda_1:\lambda_2)$ there is a
matrix $A\in\Image p$ having such a ratio of eigenvalues. Using
Lemmas~\ref{cone_conj} and \ref{linear1}, we obtain that the image
of $p$ can be either $\{0\}$, $K$,  $sl_2$, $M_2(K)$, or
$M_2(K)\setminus\tilde K.$
\end{proof}
\begin{lemma}
\label{char2} If $p$ is a multilinear polynomial evaluated on the
matrix ring $M_2(K)$, where $K$ is a quadratically closed field of
characteristic $2$, then $\Image p$ is either $\{0\}$,  $K$,
$sl_2$, or $M_2(K)$.
\end{lemma}
\begin{proof}
In view of Lemma \ref{thm1}, it suffices to assume that the image
of $p$ is $M_2(K)\setminus\tilde K$. Let
$x_1,\dots,x_m,y_1,\dots,y_m$ be generic matrices. Consider the
polynomials
$$b_i=p(x_1,\dots,x_{i-1},y_i,x_{i+1},\dots,x_m).$$ Let
$p_i(x_1,\dots,x_m,y_i)=p\tr(b_i)+\tr(p)b_i$. Hence $p_i$ can be
written as
$$p_i=p(x_1,\dots,x_{i-1},x_i\tr(b_i)+y_i\tr(p),x_{i+1},\dots,x_m).$$
Therefore $\Image p_i\subseteq\Image p.$ Also if $a\in\Image p_i$,
then
$$\tr(a)=\tr(p\,\tr(b_i)+\tr(p)\,b_i)=2\,\tr(p)\,\tr(b_i)=0.$$
Thus, $\Image p_i$ consists only of trace-zero matrices which
belong to the image of $p$. Excluding $\tilde K$, the only trace
zero matrices are nilpotent or scalar.  Thus, for each~$i,$
$p_i(x_1,\dots,x_m,y_i)$ is either scalar or nilpotent. However,
the $p_i$ are the elements of the algebra of free matrices with
traces, which is a domain.  Thus, $p_i(x_1,\dots,x_m,y_i)$ cannot
be nilpotent. Hence for all $i=1,\ldots, m,$
$p_i(x_1,\dots,x_m,y_i)$ is scalar. In this case, changing
variables leaves the plane $\langle p,I\rangle$ invariant.
Therefore, $\dim (\Image p) =  2,$ a contradiction.
\end{proof}
\begin{lemma}
\label{char_n_2} If $p$ is a multilinear polynomial evaluated on
the matrix ring $M_2(K)$ (where $K$ is a quadratically closed
field of characteristic not $2$), then $\Image p$ is either
$\{0\}$,   $K$, $sl_2$, or $M_2(K)$.
\end{lemma}
\begin{rem*}
Since the details are rather technical, we start by sketching the
proof. We assume that $\Image p=M_2(K)\setminus\tilde K.$ The
linear change of the variable in position $i$ gives us the line
$A+tB$ in the image, where $A=p(x_1,\dots,x_m)$ and
$B=p(x_1,\dots,x_{i-1},y_i,x_{i+1},\dots,x_m)$. Take the function
that maps $t$ to $f(t) = (\lambda_1-\lambda_2)^2$, where
$\lambda_i$ are the eigenvalues of $A+tB$. Evidently $$f(t) =
(\lambda_1-\lambda_2)^2=(\lambda_1+\lambda_2)^2-4\lambda_1\lambda_2=(\tr(A+tB))^2-4\det
(A+tB),$$ so our function $f$ is a polynomial of $\deg\leq 2$
evaluated on entries of $A+tB$, and thus is a polynomial in $t$.

There are three possibilities: Either $\deg_t f \leq 1,$ or $f$ is
the square of another polynomial, or $f$ vanishes at two different
values of $t$ (say, $t_1$ and $t_2$). (Note that here we use that
the field is quadratically closed). This polynomial  $f$ vanishes
if and only if the two eigenvalues of $A+tB$ are equal, and this
happens in two cases (according to Lemma \ref{thm1}): $A+tB$ is
scalar or $A+tB$ is nilpotent. Thus either both $A+t_iB$ are
scalar, or $A+t_1B$ is scalar and $A+t_2B$ is nilpotent, or both
$A+t_iB$ are nilpotent. The first case implies that $A$ and $B$
are scalars, which is impossible. The second case implies that the
matrix $A+\frac{t_1+t_2}{2} B\in\tilde K$, which is also
impossible. The third case implies that $\tr A=\tr B=0$ which we
claim is also impossible. If $\deg _t f\leq 1,$ then for large $t$
the difference $\lambda_1-\lambda_2$ of the eigenvalues of $A+tB,$
is much less than $ t,$ so the difference between eigenvalues of
$B$ must be~$0$, a contradiction.

It follows that $f(t) = (\lambda_1-\lambda_2)^2$ is the square of
a polynomial (with respect to~$t$). Thus
$\lambda_1-\lambda_2=a+tb,$ where $a$ and $b$ are some functions
of the entries of the matrices $x_1,\dots,x_m,y_i$. Note that $a$
is the difference of eigenvalues of $A$ and $b$ is the difference
of eigenvalues of $B$,  Thus
\begin{equation}\label{symmetry}
a(x_1,\dots,x_m,y_i)=b(x_1,\dots,x_{i-1},y_i,x_{i+1},\dots,x_m,x_i).
\end{equation}
Note that $(\lambda_1-\lambda_2)^2=a^2+2abt+b^2t^2$ which means
that $a^2$, $b^2$ and $ab$ are polynomials (note that here we use
$\Char K\neq 2$). Thus, $\frac{a}{b}=\frac{a^2}{ab}$ is a rational
function. Therefore there are polynomials $p_1,p_2$ and $q$ such
that $a=p_1\sqrt q$ and $b=p_2\sqrt q$. Without loss of
generality, $q$ does not have square divisors. By \eqref{symmetry}
we have that $q$ does not depend on $x_i$ and $y_i$. Now consider
the change of other variables. The function $a$ is the difference
of eigenvalues of $A=p(x_1,\dots,x_m)$ so it remains unchanged.
Thus $q$ does not depend on other variables also. That is why
$\lambda_1\pm\lambda_2$ are two polynomials and hence $\lambda_i$
are polynomials. One concludes with the last paragraph of the
proof.

\end{rem*}
\begin{proof}
According to Lemma \ref{thm1} it suffices to prove that the image
of $p$ cannot be $M_2(K)\setminus\tilde K$. Assume that the image
of $p$ is $M_2(K)\setminus\tilde K$. Consider for each
variable~$x_i$ the line $x_i+ty_i,\ t\in K$. Then
$p(x_1,\dots,x_{i-1},x_i+ty_i,x_{i+1},\dots,x_m)$ is the line
$A+tB,$ where $p(x_1,\dots,x_m)=A$ and
$p(x_1,\dots,x_{i-1},y_i,x_{i+1},\dots,x_m)=B.$ Thus
$A+tB\notin\tilde K$ for any $t$. Since $B$ is diagonalizable, we
can choose our matrix units $e_{i,j}$ such that $B$ is diagonal.
Therefore
$$B=\lambda_BI+\left(\begin{array}{cc}c & 0
\\0 & -c\end{array}\right),\ A=\lambda_AI+\left(\begin{array}{cc}x
& y \\z & -x\end{array}\right).$$ Hence
$$A+tB=(\lambda_A+t\lambda_B)I+\left(\begin{array}{cc}x+tc & y \\z
& -x-tc\end{array}\right).$$

The matrix $\left(\begin{array}{cc}x+tc & y
\\z & -x-tc\end{array}\right)$
is nilpotent if and only if $(x+tc)^2 + yz = 0,$, which has the
solution
%
%
$t_{1,2}=\frac{1}{c}(-x\pm\sqrt{-yz}).$ Thus, when $yz\neq 0$,
$\pi(A+t_jB)$ will be nilpotent for $j=1,2$, where
$\pi(X)=X-\frac{1}{2}\tr X$. However $\tr(A+t_jB)$ is nonzero for
one of these values of $t_j$, implying $A+t_jB\in\tilde K$, a
contradiction.

Thus, we must have $yz=0$. Without loss of generality we can
assume that $z=0$. Any matrix $M$ of the type
$qI+\left(\begin{array}{cc}w & h
\\0 & -w\end{array}\right)$ satisfies $\det M=q^2-w^2$ and
$q=\frac{1}{2}\tr M$. Thus $x=\sqrt{\frac{1}{4}(\tr A)^2-\det A}$
and $c=\sqrt{\frac{1}{4}(\tr B)^2-\det B}.$ Consider the matrix
$$P_i=cA-xB=p(x_1,\dots,x_{i-1},cx_i-xy_i,x_{i+1},\dots,x_m),$$
which must be scalar or nilpotent. It can be written explicitly
algebraically in terms of the entries of $x_i$ and $y_i$. Also,
$P_i=(c\lambda_B-x\lambda_A)I+(cy)e_{12}$, where $e_{12}$ is the
matrix unit. There are two cases. If $y=0$ then the line $A+tB$
includes a scalar matrix, and if $y\neq 0$ then
$(c\lambda_B-x\lambda_A)=0$ and  all matrices on the line $A+tB$
have the same ratio of eigenvalues.

Let  $S_1 = \{ i : P_i\in K\}$ and  $S_2 = \{i : P_i\in
sl_2(K)\}.$ Without loss of generality we can assume for some
$k\leq m$ that   $S_1
  =\{1,2,\dots,k\}$  and $\{k+1,\dots,m\}.$ The
four entries of $p(x_1,\dots,x_m)$ are
$$p_{ij}(x_{1,(1,1)},x_{1,(1,2)},x_{1,(2,1)},x_{1,(2,2)},\dots,x_{m,(2,2)}),$$
polynomials in the entries of $x_i$. Consider the scalar function
$$f_1(x_1,\dots,x_m)=\frac{\frac{1}{2}\tr\
p(x_1,\dots,x_m)}{R(x_1,\dots,x_m)},$$ where
$R(x_1,\dots,x_m)=\sqrt{\frac{1}{4}\tr^2p(x_1,\dots,x_m)-\det
p(x_1,\dots,x_m)}.$ This function is defined everywhere except for
those $(x_1,\dots,x_m)$ for which $p(x_1,\dots,x_m)$ is a matrix
with equal eigenvalues, because $R$ is the half-difference of
eigenvalues. The function $f_1(x_1,\dots,x_m)$ does not depend on
$x_{k+1},\dots,x_m$ because for any $i\geq k+1$, substituting
$y_i$ instead of $x_i$ does not change the ratio of eigenvalues of
$p(x_1,\dots,x_m).$ Consider the matrix function
$$f_2(x_1,\dots,x_m)=\frac{p(x_1,\dots,x_m)-\frac{1}{2}\tr\
p(x_1,\dots,x_m)}{R(x_1,\dots,x_m)}.$$ This function is also
defined everywhere except for those $(x_1,\dots,x_m)$ such that
the eigenvalues of $p(x_1,\dots,x_m)$ are equal. The function
$f_2(x_1,\dots,x_m)$ does not depend on $x_i,\ i\leq k$, because
for any $i\leq k$ substituting $y_i$ instead of $x_i$ does not
change the basis in which $p(x_1,\dots,x_m)$ is diagonal. $R^2$ is
a polynomial: $$R^2=\frac{1}{4}\tr^2p(x_1,\dots,x_m)-\det
p(x_1,\dots,x_m).$$
Write $R^2=r_1r_2r_3$ where $r_1$ is the product of all the
irreducible factors in which only $x_1,\dots,x_k$ occur, $r_2$ is
the product of all the irreducible factors  in which only
$x_{k+1},\dots,x_m$ occur, $r_3$ is the product of   the other
irreducible factors. We have that
$$\frac{\tr^2
p(x_1,\dots,x_m)}{r_1(x_1,\dots,x_m)r_2(x_1,\dots,x_m)r_3(x_1,\dots,x_m)}=f_1^2(x_1,\dots,x_m)$$
does not depend on $x_{k+1},\dots,x_m$. Therefore if $\tr^2
p=q_1q_2q_3$ (again in $q_1$ only $x_1,\dots,x_k$ occur, in $q_2$
only $x_{k+1},\dots,x_m$ occur and $q_3$ is all the rest) then
$\frac{r_1r_2r_3}{q_1q_2q_3}$ does not depend on
$x_{k+1},\dots,x_m$. Hence $r_2=q_2$ and $r_3=q_3$(up to scalar
factors). As soon as $q_1q_2q_3$ is a square of a polynomial all
$q_i$ are squares therefore $r_2$ and $r_3$ are squares. Now
consider the function $\frac{p_{12}^2}{R^2}.$ This is the square
of the  $(1,2)$-entry in the matrix function $f_2$, so it does not
depend on $x_1,\dots,x_k$. Writing $p_{12}^2=q_1q_2q_3$(where,
again,  only $x_1,\dots,x_k$ occur in $q_1$, only
$x_{k+1},\dots,x_m$ occur in $q_2$  and $q_3$ is comprised of all
the rest), then all the $q_i$ are squares and $q_1=r_1$, implying
$r_1$ is square. Thus the polynomial $r_1r_2r_3=R^2$ is the square
of a polynomial. Therefore $R$ is a polynomial. We conclude that
$\lambda_1-\lambda_2=2R$ is a polynomial (where we recall that
$\lambda_1$ and $\lambda_2$ are the eigenvalues of
$p(x_1,\dots,x_m)$). $\lambda_1+\lambda_2=\tr (p)$ is also a
polynomial and hence $\lambda_i$ are polynomials, which obviously
  are invariant under conjugation since
  any conjugation is the square of some other conjugation).
Hence, $\lambda_i$ are the polynomials of traces, by Donkin's
Theorem quoted above. Now consider the polynomials
$(p-\lambda_1I)$ and $(p-\lambda_2I),$ which are   elements of the
algebra of free matrices with traces, which we noted above is a
domain. Both are not zero but their product is zero, a
contradiction.
\end{proof}
Finally, Theorem \ref{main} follows from Lemmas \ref{char2} and
\ref{char_n_2}.

\end{document}